\documentclass[10pt]{article}
\usepackage{ amssymb, amsmath,amsthm, hyperref,
  color, tikz ,  geometry, fancyhdr, mathtools}
  \usepackage{mathabx}
\usepackage[all,2cell]{xy}
   \UseAllTwocells
   \SilentMatrices
\usepackage[dvips]{epsfig}

\usetikzlibrary{arrows, decorations.markings, decorations, patterns,
positioning, decorations.pathreplacing, decorations.pathmorphing,
external, fit, calc}
\tikzset{->-/.style={decoration={markings, mark=at position .5 with {\arrow{>}}},postaction={decorate}}}
\tikzset{-<-/.style={decoration={markings, mark=at position .5 with {\arrow{<}}},postaction={decorate}}}

\usetikzlibrary{arrows, decorations.markings, decorations, patterns,
positioning, decorations.pathreplacing, decorations.pathmorphing, external}
\tikzset{->-/.style={decoration={markings, mark=at position .5 with {\arrow{>}}},postaction={decorate}}}
\tikzset{-<-/.style={decoration={markings, mark=at position .5 with {\arrow{<}}},postaction={decorate}}}

\newcommand{\NB}[1]{\ensuremath{\vcenter{\hbox{#1}}}}

\normalfont\upshape

\theoremstyle{definition}
\newtheorem{thm}{Theorem}[section]
\newtheorem{cor}[thm]{Corollary}

\newtheorem{lem}[thm]{Lemma}

\newtheorem{prop}[thm]{Proposition}

\newtheorem{example}[thm]{Example}

\newtheorem*{thm*}{Theorem}

\numberwithin{equation}{section}

\def\o{\otimes}

\def\N{{\mathbb N}}
\def\R{{\mathbb R}}
\def\Z{{\mathbb Z}}

\newcommand{\End}{{\rm End}}

\def\M{\mathrm{M}}

\def\dif{\partial}

\def\lra{{\longrightarrow}}
   
           \def\gdim{{\mathrm{gdim}}}

\def\Id{\mathrm{Id}}

        \def\shuffle{\,\raise 1pt\hbox{$\scriptscriptstyle\cup{\mskip
               -4mu}\cup$}\,}

\newcommand{\refequal}[1]{\xy {\ar@{=}^{#1}
(-1,0)*{};(1,0)*{}};
\endxy}

\newcommand{\mH}{\mathrm{H}} 
\newcommand{\Ul}{{\mathsf U}}

\newcommand{\ul}{{\mathsf u}}

\newcommand{\mC}{\mathrm{C}}

\renewcommand{\bar}[1]{\overline{#1}}

\title{Remarks on some infinitesimal symmetries of Khovanov--Rozansky homologies in finite characteristic}
\author{You Qi, Louis-Hadrien Robert, Joshua Sussan and Emmanuel Wagner}

 \hypersetup{
    pdftoolbar=true,        
    pdfmenubar=true,        
    pdffitwindow=false,     
    pdfstartview={FitH},    
    pdftitle={Remarks on symmetric link homologies},    
    pdfauthor={You Qi, Louis-Hadrien Robert, Joshua Sussan and Emmanuel Wagner},     
    pdfsubject={Remarks on symmetric link homologies},   
    pdfcreator={You Qi, Louis-Hadrien Robert, Joshua Sussan and Emmanuel Wagner},   
    pdfproducer={You Qi, Louis-Hadrien Robert, Joshua Sussan and Emmanuel Wagner}, 
    pdfkeywords={}, 
    pdfnewwindow=true,      
    colorlinks=true,       
    linkcolor=darkgray,          
    citecolor=teal,        
    filecolor=magenta,      
    urlcolor=violet,          
    linkbordercolor=red,
    citebordercolor=teal,
    urlbordercolor=violet,  
    linktocpage=true
    }
\date{\today}

\begin{document}

\maketitle

\begin{abstract}
We give a new proof of a theorem due to Shumakovitch and Wang on base point independence of Khovanov--Rozansky homology in characteristic $p$. Some further symmetries of $\mathfrak{gl}(p)$-homology in characteristic $p$ are also discussed.
\end{abstract}

\setcounter{tocdepth}{2} \tableofcontents

\section{Introduction}

In order to understand the structure of link homologies, various symmetries of these invariants have been studied.  Khovanov and Rozansky defined an action of a partial Witt algebra on triply-graded homology \cite{KRWitt}.  This action was studied in characteristic $p$ in order to categorify link polynomials at prime roots of unity \cite{QiSussanLink, QRSW1}.  Through the apparatus of foam evaluations \cite{RW1}, the authors constructed a partial Witt action on equivariant $\mathfrak{gl}(N)$-foams and extracted an $\mathfrak{sl}(2)$-action on equivariant $\mathfrak{gl}(N)$-link homology \cite{QRSW2, QRSW3} where the ground field does not have characteristic two.  These papers were partially motivated by a paper of Wang \cite{WangJ} extending earlier work of Shumakovitch \cite{Schum} which proved that reduced $\mathfrak{gl}(p)$-link homology in characteristic $p$ is independent of base point.

In \cite{QRSW3} we showed that if one works in characteristic $p$, then the $\mathfrak{sl}(2)$-action descends to the non-equivariant link homology obtained by killing the equivariant parameters.  The operators $e$ and $f$ then could be viewed as $p$-differentials in this finite characteristic case.
In this note we give a new proof of Shumakovitch and Wang's result
using $p$-DG machinery.  We study some of the structural consequences
of the $\mathfrak{sl}(2)$-action here.  In particular, we show that
the homology of a web has a unimodality property.  We also prove that a slice knot must contain the Steinberg representation of $\mathfrak{sl}(2)$ as a summand.  Finally, we show that as a consequence of a theorem of Wang \cite{WangJ2}, that $\mathfrak{sl}(2)$-actions could detect the splitness of a link.

The $\mathfrak{sl}(2)$-actions we consider here are also connected to \cite{EliasQislact, KhovNote} where such actions are constructed on categorified quantum groups and categories of Soergel bimodules. Cooper and Beliakova considered Steenrod operations on nilHecke algebras \cite{BeCoSteenrod} which have a flavor of the actions considered here.

The symmetries considered in this work are different from the ones in \cite{GHM} where the symmetries act in a homological direction.  
We also note that our $\mathfrak{sl}(2)$-actions on $\mathfrak{gl}(p)$-homology are different from the $\mathfrak{sl}(N)$-actions on annular
$\mathfrak{gl}(N)$-homology considered in \cite{GLW, QRS}.  Those actions also do not require working over a field of finite characteristic.

\paragraph{Acknowledgements.}
Y.~Q. is partially supported by the Simons Foundation and the National Science Foundation (DMS-2401376).  J.~S.{} is
partially supported by the NSF grant DMS-2401375, a Simons Foundation Travel Support Grant and PSC CUNY Enhanced
Award 66685-00 54. E.~W.{} is partially supported by the ANR projects AlMaRe
(ANR-19-CE40-0001-01), AHA (JCJC ANR-18-CE40-0001) and CHARMES
(ANR-19-CE40-0017). L-H.~R.{} was not supported by any project related
funding.

\section{A new proof of Shumakovitch--Wang's Theorem}

\subsection{A \texorpdfstring{$p$}{p}-DG  Frobenius algebra} \label{sec-DG-Frob} In this section we collect together some background material we need.

Let $\Bbbk$ be a field of characteristic $p>0$. Unadorned tensor product ``$\otimes$'' will be understood as taken over $\Bbbk$.

Consider the graded Frobenius algebra $A=\Bbbk[x]/(x^p)$ where the degree of $x$ is two. The comultiplication $\Delta: A \lra A\otimes A $ is an $(A,A)$-bimodule homomorphism determined by
\begin{equation}
    \Delta(1)=\sum_{i=0}^{p-1} x^i \otimes x^{p-1-i}, 
\end{equation}
The counit $\epsilon: A\lra \Bbbk$ is given by
\begin{equation}
    \epsilon(x^i) =
    \begin{cases}
        1& i= p-1, \\
        0 & i \neq p-1. \\
    \end{cases}  
\end{equation}

Let us equip $A$ with $p$-differential of degree $-2$. This $p$-differential on the Khovanov chain complex of a link was first utilized by Shumakovitch \cite{Schum} in the characteristic 2 case. A characteristic $p$ analogue was more recently considered by \cite{WangJ} to prove analogous results of \cite{Schum} in characteristic $p>2$. For this reason, we will call this differential as the \emph{Shumakovitch--Wang differential}, and denote it by $\dif_-$.  This is defined by $\dif_{-}: A\lra A$:
\begin{equation}
    \dif_- (x^i)=-i x^{i-1}, \quad i=0,\dots, p-1.
\end{equation}
In other words, $\dif_-$ has the effect of the differential operator $-\frac{\partial}{\partial x}$ on the truncated polynomial algebra $A$.

For any non-negative integer $n\in \N$, identify $A^{\otimes n}\cong \Bbbk[x_1,\dots, x_n]/(x_1^p,\dots, x_n^p)$. The differential $\dif_-$ extends to a differential on $A^{\otimes n}$ according to the Leibniz rule. In other words, the $\dif_-$-action is given by the differential operator $\dif_-=-(\frac{\dif}{\dif x_1}+\cdots + \frac{\dif}{\dif x_n})$.

We refer the reader to \cite{QiSussan2} for a survey of standard results about $p$-DG algebras. For this paragraph only, let $(A^\prime,\dif_A)$ be an arbitrary $p$-DG algebra, and $H$ be the Hopf algebra $\Bbbk[\dif]/(\dif^p)$. Set $B:=A^\prime\# H$ to be the \emph{smash product algebra}, which is isomorphic, as a vector space, to $A^\prime\otimes H$, subject to the relations that $A^\prime\cong A^\prime\o 1\subset A^\prime\o H$, $H\cong  1\o H \subset A^\prime\o H$ as subalgebras, and the rule for commuting elements
\begin{equation}\label{eqn-smash}
(1\o \dif) \cdot (a\o 1 )=(a\o 1) \cdot(1\o \dif) + \dif_A(a)\o 1,
\end{equation}
where $a$ is an arbitrary element of $A^\prime$. Clearly, equation \eqref{eqn-smash} is homogeneous. Thus $B$ is a graded algebra, and it is an easy exercise to see that the category of graded $B$-modules coincides with the category of $p$-DG modules over $(A^\prime,\dif_A)$.

\begin{lem}
The Frobenius algebra structure on $A$ is compatible with the differential action by $\dif_-$. Consequently, $(A,\dif_-)$ is a contractible $p$-DG algebra.
\end{lem}
\begin{proof}
It suffices to show that the Frobenius structural maps $\Delta$ and $\epsilon$ commute with the differential actions. For instance, we have, on one hand,
\[
\Delta(\dif_-(1)) =\Delta(0)=0,
\]
and on the other hand,
\begin{align*}
 \dif_-(\Delta(1))& =
\dif_- \left(\sum_{i=0}^{p-1}  x^{i}\otimes x^{p-1-i}\right)=
 -\sum_{i=0}^{p-1} i x^{i-1}\otimes x^{p-1-i} - \sum_{i=0}^{p-1} (p-1-i) x^{i}\otimes x^{p-2-i} \\
 & = -\sum_{i=0}^{p-2} (i+1) x^{i}\otimes x^{p-i-2} -
 \sum_{i=0}^{p-2} (p-1-i) x^{i}\otimes x^{p-1-i}  = 0.
\end{align*}
The rest of the compatibility checks are similar.
\end{proof}

In characteristic $2$, a simple computation shows that there is an isomorphism
\begin{equation}
    \phi:B \cong \mathrm{M}(2,\Bbbk),
\end{equation}
given by 
    \[
\phi(1\otimes \dif_-)= \left( 
        \begin{matrix}
        0 & 0 \\
        1 & 0
        \end{matrix}
        \right) , \quad \quad
        \phi(x\otimes 1)=
        \left( 
        \begin{matrix}
        0 & 1 \\
        0 & 0
        \end{matrix}
        \right), 
        \quad \quad
        \phi(1\otimes 1)=
        \left( 
        \begin{matrix}
        1 & 0 \\
        0 & 1
        \end{matrix}
        \right) ,
        \quad \quad
        \phi(x\otimes \dif_-)=
        \left( 
        \begin{matrix}
        1 & 0 \\
        0 & 0
        \end{matrix}
        \right).
    \]
This is true more generally.

\begin{lem}\label{lem-2-by-2-matrix}
  For the $p$-DG algebra $(A,\dif_{-})$ over a field $\Bbbk$ of positive characteristic $p$, there is an isomorphism of graded algebras
    \begin{equation}
    \phi: B \lra \mathrm{M}(p,\Bbbk).
    \end{equation} 
Here the grading on the matrix algebra is given by the principal grading 
    \begin{equation*}
        \mathrm{deg}(E_{i,j})=2(j-i)
    \end{equation*}
\end{lem}
\begin{proof}
To avoid confusing notation, let $V$ be the $p$-DG module $A\cdot v_0$ where $\dif(v_0)=0$. Thus it is naturally a $B$-module so that we have a homomorphism $\phi: B\lra \End_\Bbbk(V)\cong \M(p,\Bbbk)$. It suffices to show that $\phi$ is an isomorphism, and that the gradings match on both sides. 

To show that $\phi$ is an isomorphism, it suffices to show that it is injective, since both sides have dimension $p^2$. This is clear since $B$ has a basis consisting of $\{x^i\partial^j|0\leq i,j\leq p-1\}$, and this basis acts on $V$ linearly independently.

\end{proof}

\begin{cor}\label{cor-acyclicity}
Any $p$-DG module over $(A,\dif_-)$ is isomorphic to a direct sum of degree shifted copies of the column module over $\mathrm{M}(p,\Bbbk)$. Consequently, any $p$-DG module over $(A,\dif_{-})$ is acyclic.
\end{cor}
\begin{proof}
By the earlier discussion, $p$-DG modules over $(A,\dif_{-})$ are equivalent to graded modules over $\mathrm{M}(p,\Bbbk)$, and thus the result is a special case of the classical Wedderburn Theorem.

The second statement follows since any graded column module over $\mathrm{M}(p,\Bbbk)$ is a contractible $p$-complex.
\end{proof}

In what follows, we will study the rank-one free module $A$ as an $(A,A)$-bimodule. In particular, we will need to know how the bimodule structure interacts with the differential.

\begin{lem}\label{lem-tensor-smash}
     For the tensor product $p$-DG algebra $(A\otimes A,\dif_-)$,  its  smash product algebra $(A\otimes A)\# H$ is isomorphic to
    $A\otimes \mathrm{M}(p,\Bbbk)$.
    More generally, $A^{\otimes n}\# H$ is isomorphic to the tensor product algebra $A^{\otimes (n-1)}\otimes \mathrm{M}(p,\Bbbk)$.
\end{lem}
\begin{proof}
    Identify $A\otimes A =\Bbbk[x_1,x_2]/(x_1^p,x_2^p)$. Then the $\dif_-$-action annihilates $1\otimes 1$ and $y:=x_1- x_2$. Thus we have
    \[
    (A\otimes A)\# H \cong \dfrac{\Bbbk[y]}{(y^p)} \otimes \left(\dfrac{\Bbbk[ x_1 ]}{( x_1^p)}\# H\right).
    \]
    More generally, identify $A^{\otimes n}$ with $\Bbbk[x_1,\dots, x_n]/(x_1^p,\dots, x_n^p)$. The new variables $y_i:=x_1-x_i$, $i=2,\dots, n$, are annihilated by the $\dif_-$-action and become zero when raised to $p$th power. Thus we have
    \[
    A^{\otimes n}\# H \cong \dfrac{\Bbbk[y_2,\dots, y_{n}]}{(y_2^p,\dots, y_{n}^p)} \otimes \left(\dfrac{\Bbbk[ x_1 ]}{( x_1^p)}\# H\right).
    \]
    The result follows.
  \end{proof}

\subsection{A proof of Shumakovitch--Wang's Theorem} 

Let $D$ be a link diagram. We form its $\mathfrak{gl}(p)$-Khovanov--Rozansky chain complex via
foam evaluation \cite{RW1}. Fix a base point $b\in D$.  The bigraded Khovanov--Rozansky chain complex $C_{*,*}(D)$ carries an action of the algebra $A=\Bbbk[x_b]/(x_b^p)$ which is locally given by merging a small unknot near $b$. Furthermore, the chain complex also has an action of $\dif_-$ \cite{Schum, WangJ, QRSW2, QRSW3} that commutes with the
topological differential $d_T$. In other words, $C_{*,*}(D)$ is a chain complex of $p$-DG modules over $A$.

Recall that $B=A\#H$ is a graded matrix algebra over $\Bbbk$ (Lemma \ref{lem-2-by-2-matrix}). Graded chain complexes of $B$-modules decompose, up to $q$-degree and homological degree shifts, into direct sums of modules of the form
\begin{subequations}
\begin{equation}\label{eqn-indecomp-mod-over-B-1}
    0\lra V \lra 0, 
\end{equation}
\begin{equation}\label{eqn-indecomp-mod-over-B-2}
     0\lra V\stackrel{=}{\lra} V \lra 0,
\end{equation}
\end{subequations}
where $V$ is a simple (column) module over $B$.

Next, we have an isomorphism
\begin{equation}
  \psi:  A\otimes_B V \stackrel{\cong}{\lra}\Bbbk, 
\end{equation}
This isomorphism is part of the usual graded Morita equivalence between a matrix algebra and the ground field. For simplicity, we will usually suppress this isomorphism $\psi$, and write $x\otimes_B 1$ instead for $\psi(x\otimes 1)$, and call $A\otimes_B V$ the \emph{Morita reduction}. 

We set, for the chosen base point $b$, the  chain complex obtained by reducing each term in a fixed topological degree by Morita reduction
 \[
    \bar{C}_{*,*}(D):=C(D)_{*,*}\otimes_B V,
    \]
    equipped with the induced topological differential $\bar{d}_T:=d_T\otimes_B \Id_V$.
When no confusion can arise, we abbreviate $C_{*,*}(D)$
 and $\bar{C}_{*,*}(D)$ simply as $C(D)$ and $\bar{C}(D)$. Their total homology groups are denoted $\mH(D)$ and $\bar{\mH}(D)$ respectively.

 \begin{lem}
    The Morita reduced chain complex $(\bar{C}_{*,*}(D),\bar{d}_T)$ is isomorphic to the reduced Khovanov--Rozansky chain complex of $D$.
\end{lem}
\begin{proof}
    Regarding $A$ as a subalgebra of $B=A\# H$, we have
    \[V\cong B\otimes_A \Bbbk\]
    where $\Bbbk$ is the simple module over $A$ (the base point variable $x_b$ acts by 0). Thus, we have an isomorphism of functors
   \[ (\mbox{-})\otimes_BV \cong  (\mbox{-})\otimes_B \otimes B\otimes_A \Bbbk\cong  (\mbox{-})\otimes_A \Bbbk.\]
   This functor is termwise applied to $C(D)$ in each homological degree.
\end{proof}

\begin{prop}\label{prop-unreduced-reduced}
    The chain complex $C(D)$ is a bigraded free module over $A$: 
    \[
    C(D)\cong \bar{C}(D)\otimes V.
    \]
    Furthermore, the $\mathfrak{gl}(p)$-Khovanov--Rozansky homology of $D$ is a free $A$-module
       \[
    \mH(D)\cong \bar{\mH}(D)\otimes V.
    \]
\end{prop}
\begin{proof}
The first statement follows from \eqref{eqn-indecomp-mod-over-B-1} and \eqref{eqn-indecomp-mod-over-B-2}.
     The second part follows from the fact that taking homology of $C(D)$ precisely eliminates contractible summands of the form \eqref{eqn-indecomp-mod-over-B-2}, which is also eliminated in $\bar{\mH}(D)$:
    \[
    \left(  0\lra V\stackrel{=}{\lra} V \lra 0 \right) \otimes_BV \cong\left(
  0\lra \Bbbk\stackrel{=}{\lra}\Bbbk \lra 0 \right).
    \]
    The simple summands that contribute to homology reduce to
    \[
\left(  0\lra V \lra 0 \right) \otimes_BV \cong\left(
  0\lra \Bbbk \lra 0 \right).
    \]
    The result follows.
\end{proof}

For the trivial unlink diagram with two components, the $\mathfrak{gl}(p)$-Khovanov--Rozansky chain complex assigned to it is $A^{\otimes 2}\cong \Bbbk[x_1,x_2]/(x_1^p,x_2^p)$. In characteristic $p$, setting $y=x_1-x_2$ on which $\dif_-$ acts by zero, we have the $p$-DG module isomorphism
\[\dfrac{\Bbbk[x_1,x_2]}{(x_1^p,x_2^p)}\cong \dfrac{\Bbbk[x_1,y]}{(x_1^p,y^p)} \cong \dfrac{\Bbbk[y,x_2]}{(y^p,x_2^p)}.\]
Thus reduction at either $x_1$ or $x_2$ produces the $p$-dimensional complex $\Bbbk[y]/(y^p)$, which is independent of the chosen base point. This theorem, due to Shumakovitch \cite{Schum} and Wang \cite{WangJ}, is true more generally, which we give a simplified proof of now.

\begin{thm}
Let $D$ be the diagram of a knot in $\R^3$. The reduced complex of $\bar{C}(D)$ is independent of choices of base points.
\end{thm}

\begin{proof}
Let $x_1$, $x_2$ be the corresponding variables for two different base points $b_1$, $b_2$ on the link diagram $D$. Merging with two small circles at these base points defines a $p$-DG module structure on $C(D)$ over the $p$-DG algebra $\Bbbk[x_1,x_2]/(x_1^p,x_2^p)$. Setting $y=x_1-x_2$, by Lemma \ref{lem-tensor-smash}, we have isomorphisms
\[
A^{\otimes 2}\# H \cong \dfrac{\Bbbk[x_i,y]}{(x_i^p,y^p)}\# H \cong A_y\otimes B_i, \quad \quad i=1,2.\]
Here we have identified $A_y=\Bbbk[y]/(y^p)$ and the matrix algebras $B_i\cong\mathrm{M}(p,\Bbbk)$, $i=1,2$, with
\[
B_i \cong \dfrac{\Bbbk[x_i]}{(x_i^p)}\# H.
\]
Under the above identifications, the reduced chain complexes at $x_1$
and $x_2$, denoted temporarily by $\bar{C}_1(D)$ and $\bar{C}_2(D)$,
are both isomorphic to
\[
\bar{C}_i(D):= C(D)\otimes_{B_i} \left( \dfrac{\Bbbk[x_i]}{(x_i^p)} \right)\cong
C(D)\otimes_{A_y\otimes B_i} \left( \dfrac{\Bbbk[y,x_i]}{(y^p, x_i^p)}\right)
\cong
C(D)\otimes_{A^{\otimes 2}\# H} \left( \dfrac{\Bbbk[x_1,x_2]}{(x_1^p, x_2^p)}\right).
\]
One can also check that the isomorphism is compatible with topological differentials by the functoriality of reductions. The theorem follows by identifying the complexes $\bar{C}_1(D)$ and $\bar{C}_2(D)$ through this chain of isomorphisms.
\end{proof}

 \section{Infinitesimal \texorpdfstring{$\mathfrak{sl}(2)$}{sl(2)} symmetries of \texorpdfstring{$\mathfrak{gl}(p)$}{gl(p)}-homology}

\subsection{Background on graded representation theory of restricted \texorpdfstring{$\mathfrak{sl}(2)$}{sl(2)}}
In \cite{QRSW2, QRSW3}, the differential operator $\dif_-$ is incorporated into a larger Lie algebra action on  $\mathfrak{gl}(N)$-equivariant foams and conjecturally Khovanov--Rozansky homology. Namely, one-half of the usual Witt algebra acts on the state space associated to closed webs and conjecturally on link homology groups. Of particular interest to us is a copy of $\mathfrak{sl}(2)$ sitting inside the Witt algebra which does act on link homology. Over a field of characteristic $p\geq 3$, the usual generators of $\mathfrak{sl}(2)$ correspond to the differential operators (their usual notation in parenthesis)
\begin{equation}
    \dif_-=-\frac{\dif}{\dif x} (=e), \quad \quad \dif_0 = -2x\frac{\dif}{\dif x}(=h), \quad \quad
    \dif_+ = x^2 \frac{\dif}{ \dif x}(=f).
\end{equation}
In this subsection, we gather together some well-known background information on graded representation theory of $\mathfrak{sl}(2)$ in characteristic $p$. Throughout, we will assume that $p$ is a prime greater than $2$ and $\Bbbk$ is a field of characteristic $p$. 

Let $\Ul(\mathfrak{sl}(2))$ and $\ul(\mathfrak{sl}(2))$ be the usual universal enveloping and restricted universal enveloping algebras of $\mathfrak{sl}(2)$ respectively. Namely, as an associative algebra
\begin{equation}
    \ul(\mathfrak{sl}(2))=\Ul(\mathfrak{sl}(2))/(\dif_+^p, \dif_0^p-\dif_0, \dif_-^p).
\end{equation}
When no confusion can arise, we adopt the simplified notation $\Ul$ and $\ul$ for $\Ul(\mathfrak{sl}(2))$ and $\ul(\mathfrak{sl}(2))$ respectively. Both algebras are $\Z$-graded, with the degree convention fixed by
\begin{equation}\label{eqn-differential-grading}
    \mathrm{deg}(\dif_{\pm})=\pm 2,\quad \quad
    \mathrm{deg}(\dif_0)=0.
\end{equation}
We will also utilize Borel subalgebras of $\mathfrak{sl}(2)$ defined by 
\begin{equation}
    \mathfrak{b}_{+}:=\Bbbk\langle \dif_+, \dif_0\rangle \quad \quad \quad
        \mathfrak{b}_{-}:=\Bbbk\langle \dif_-, \dif_0\rangle.
\end{equation}
 Let $\Ul_\pm$ and $\ul_{\pm}$ be the usual distribution
  and restricted universal enveloping algebras of $\mathfrak{b}_{\pm}$ respectively. They are naturally $\Z$-graded according to \eqref{eqn-differential-grading}.

As a convention, when $A = \oplus_{i\in \Z} A_i$ is a graded algebra, and $M=\oplus_{i\in \Z} M_i$ is a graded $A$-module in the sense that $A_iM_j\subset M_{i+j}$, we define, for any $k\in \Z$,
\begin{equation}
    q^k M:= \oplus_{j\in \Z}(q^k M)_j, \quad \quad
    (q^kM)_j=M_{j-k}.
\end{equation}
Namely $q^kM$ is has the same underlying $A$-module structure as $M$, but its degree is shifted up by $k$.

 For the integers $\lambda\in \{0,\dots, p-1\}$, there is a balanced simple $\ul$-module $L(\lambda)$ of dimension $\lambda+1$ defined similarly as in characteristic zero. More explicitly, set $L(0)=\Bbbk$ to be the trivial $\ul$-module, and consider $L(1)=\Bbbk \langle w_{\pm 1} \rangle$, where the nonzero actions of $\dif_0$, $\dif_\pm$ are given by
 $$
 \dif_0(w_{\pm 1})=\pm w_{\pm 1}, \quad \quad
 \dif_+(w_{-1})=w_{1}, \quad \quad 
 \dif_-(w_1)=w_{-1}.
 $$
 Then, for $\lambda \in \{1,\dots, p-1\}$, define $L(\lambda):= S^\lambda(L(1))$, where the notation
denotes the polynomial space of degree $\lambda$ in a tensor power of
$L(1)$. One can readily check that such $L(\lambda)$'s are
simple. Moreover, any graded simple $\ul$-module is obtained as a
degree shifted copy of $L(\lambda)$ as follows. For $\mu \in \Z$,
write by Euclidean division (with remainder) that
$\mu=\mu_0+p\mu_1$ where $\mu_0\in \{0,\dots, p-1\}$ and define:
\[
L(\mu): =q^{-p\mu_1} L(\mu_0).
\]

 We also recall another class of important $\ul$-modules. For $\lambda \in \Z$, define 
 $\Delta(\lambda)$ to be the $\ul$-module as follows. Consider the 1-dimensional $\mathfrak{b}_-$-
 module $\Bbbk_\lambda=\Bbbk u_0$, where $\mathrm{deg}(u_0)=\lambda$, $\dif_0$ acts on $u_0$ by the scalar $\lambda$ and $\dif_-(u_0)=0$. Set
 \[
 \Delta(\lambda):= \ul\otimes_{\ul_-}\Bbbk_\lambda.
 \]
 The usual PBW Theorem implies that the module $\Delta(\lambda)$ is free over $\ul_+\cong \Bbbk[\dif_+]/(\dif_+^p)$. A basis of this module is given by 
 $$\left\{u_i:=\frac{1}{i!}\dif_+^i\otimes u_0| 0\leq i \leq p-1\right\} .$$ 
 The explicit $\mathfrak{sl}(2)$-action on $\Delta(\lambda)$ on the basis is given by 
\[
\dif_0 \cdot u_i = (\lambda-2i)u_i, \quad \quad
\dif_- \cdot u_i = (\lambda-i+1) u_{i-1}, \quad \quad
\dif_+ \cdot u_i = (i+1) u_{i+1} \ .
\]
 
Let $\nabla(\lambda)$ be the $\mathfrak{sl}(2)$-module with basis $\{v_0,\ldots,v_{p-1}\}$ with action given by
\[
\dif_0 \cdot v_i = (\lambda-2i)v_i, \quad \quad
\dif_- \cdot v_i = -i v_{i-1}, \quad \quad
\dif_+ \cdot v_i = (-\lambda+i) v_{i+1} \ .
\]

The module $\Delta(\lambda)$ is known as the \emph{baby Verma modules of highest weight $\lambda$}, while the module $\nabla(\lambda)$ is known as the \emph{dual baby Verma modules of highest weight $\lambda$}. Schematically, we depict the $\mathfrak{sl}(2)$ action as follows.
\begin{subequations} \label{vermacovermapictures} 
\begin{equation}
\Delta(\lambda): \quad
\xymatrix{
u_0 \ar@(ul,ur)^{\dif_0=\lambda} \ar@/^/[rr]^{\dif_+=1} && u_1\ar@(ul,ur)^{\dif_0=\lambda-2} \ar@/^/[ll]^{\dif_-=\lambda} \ar@/^/[rr]^{\dif_+=2}  && \cdots \ar@/^/[ll]^{\dif_-=\lambda-1}  \ar@/^/[rr]^{\dif_+=p-2} && u_{p-2}\ar@(ul,ur)^{\dif_0=\lambda-2p+4}\ar@/^/[ll]^{\dif_-=\lambda-(p-3)}  \ar@/^/[rr]^{\dif_+=p-1} && u_{p-1}\ar@(ul,ur)^{\dif_0=\lambda-2p+2}\ar@/^/[ll]^{\dif_-=\lambda-(p-2)} 
}
\end{equation}
\begin{equation}
\nabla(\lambda): \quad
\xymatrix{
v_0 \ar@(ul,ur)^{\dif_0=\lambda} \ar@/^/[rr]^{\dif_+=-\lambda} && v_1
\ar@(ul,ur)^{\dif_0=\lambda -2} \ar@/^/[ll]^{\dif_-=p-1}
\ar@/^/[rr]^{\dif_+=1-\lambda}  && \cdots \ar@/^/[ll]^{\dif_-=p-2}
\ar@/^/[rr]^{\dif_+=-3-\lambda} && v_{p-2}
\ar@(ul,ur)^{\dif_0=\lambda-2p+4} \ar@/^/[ll]^{\dif_-=2}
\ar@/^/[rr]^{\dif_+=-2-\lambda} && v_{p-1} \ar@(ul,ur)^{\dif_0=\lambda
-2p +2}\ar@/^/[ll]^{\dif_-=1} 
}
\end{equation}
\end{subequations}

For $\lambda=p-1$, the modules $L(p-1)$,$\Delta(p-1)$, and $\nabla(p-1)$ are isomorphic. This module is usually referred to as
the \emph{Steinberg module} for $\ul$.

Since $\ul$ is a graded Hopf algebra, there is a natural grading-reversing duality on $\ul$-modules, which we denote by $(\mbox{-})^*$: Given a $\ul$-module $M$, and any $x\in \mathfrak{sl}(2)$, $x$ acts on dual vector space $M^*$ of $M$ by
\begin{equation}\label{eqn-dual-rep}
(x\cdot f)(v):=-f(x\cdot v),
\end{equation}
for any $f\in M^*$ and $v\in M$. Under this duality, it is an easy exercise to show that, for any $\lambda\in \Z$,
\begin{equation}
\Delta(\lambda)^* \cong \Delta(2p-2-\lambda), \quad \quad
\nabla(\lambda)^* \cong \nabla(2p-2-\lambda).
\end{equation}

There is another important duality in the representation theory of $\ul$, although we will not use it in this paper. Recall that there is a degree-reversing Cartan involution $\omega \colon \ul \rightarrow \ul$ determined by
\[
\dif_+ \mapsto \dif_- \ , \quad \dif_- \mapsto \dif_+ \ , \quad
\dif_0 \mapsto - \dif_0 \ .
\]
Under this automorphism, taking the vector space dual on a
representation followed by twisting the $\ul$-action by $\omega$
exchanges baby Verma modules and dual baby Verma modules as follows 
\[
\Delta(\lambda)^\omega \leftrightsquigarrow \nabla(2p-2-\lambda) \ .
\]
Indecomposable graded projective-injective modules are also labeled by $\lambda\in
\Z$. A theorem of Bellamy--Thiel (\cite[Theorem
5.1]{BeTh2}) shows that when a graded $\ul$-module has simultaneous
$\Delta$ and $\nabla$ filtrations\footnote{By $\Delta$ and $\nabla$
  filtrations, we mean, filtrations by baby Verma modules and by dual
  baby Verma modules respectively.}, then it is graded projective-injective, and is a direct sum of graded indecomposable projective-injective modules that we next describe. For now, notice that, if $\lambda=kp-1$ for all $k\in \Z$, then 
\begin{equation}
    \Delta(\lambda)\cong\nabla(\lambda)\cong L(\lambda),
\end{equation} and thus it is not only simple, but also
projective-injective by Bellamy--Thiel's Theorem. 

Indecomposable graded projective modules have been extensively studied. For instance, in \cite{GeKa}, they are constructed as follows. For $\lambda\in \{0,\dots, p-2\}$, form the tensor product $L(p-1-\lambda)\otimes L(p-1)$ and project it onto the direct summand containing the highest weight vector. The summand obtained is $2p$-dimensional, with the nonzero $\dif_{\pm}$ action on weight vectors depicted in the diagram (the $\dif_0$-action has the effect of counting the $q$-degree mod $p$, and is thus omitted here):
\begin{equation}\label{eqn-indecomposable-projective}
\begin{gathered}
\xymatrix@C=0.75em{
&& 2p-\lambda-2 & \dots & \lambda+2 & \lambda & \lambda-2 & \dots & 2-\lambda & -\lambda & -2-\lambda & \dots & \lambda+2-2p\\
&& & & & \bullet \ar[dl]^{\dif_-} \ar[ddr]^{\dif_+} \ar@/^/[r]^{\dif_+} & \bullet\ar@/^/[l]^{\dif_-} \ar@/^/[r]^{\dif_+} \ar[ddr]^{\dif_+}   & \cdots \ar@/^/[l]^{\dif_-} \ar@/^/[r]^{\dif_+} & \bullet \ar[ddr]^{\dif_+}   \ar@/^/[l]^{\dif_-} \ar@/^/[r]^{\dif_+} & \bullet \ar@/^/[l]^{\dif_-}\ar[dr]^{\dif_+}\\
P(\lambda): && \bullet \ar@/^/[r]^{\dif_+} & \cdots \ar@/^/[l]^{\dif_-} \ar@/^/[r]^{\dif_+}  & \bullet \ar@/^/[l]^{\dif_-} \ar[dr] ^{\dif_+}&&&&&&\bullet \ar[dl]^{\dif_-} \ar@/^/[r]^{\dif_+} & \cdots\ar@/^/[l]^{\dif_-}  \ar@/^/[r]^{\dif_+} & \bullet \ar@/^/[l]^{\dif_-} \\
&& & & & \bullet  \ar@/^/[r]^{\dif_+} & \bullet\ar@/^/[l]^{\dif_-} \ar@/^/[r]^{\dif_+} & \cdots \ar@/^/[l]^{\dif_-} \ar@/^/[r]^{\dif_+} & \bullet \ar@/^/[l]^{\dif_-} \ar@/^/[r]^{\dif_+} & \bullet \ar@/^/[l]^{\dif_-}
}
\end{gathered}
\end{equation}
More generally, if $\lambda=\mu_0+p\mu_1$, with $\mu_0\in \{0,\dots, p-2\}$, then $q^{-p\mu_1}P(\mu_0)$ is the indecomposable projective cover and injective envelope of the simple module $L(\lambda)$. Together with $L(kp-1)$, $k\in \Z$, the modules $P(\lambda)$, $\lambda\in \Z\backslash \{kp-1|k\in \N\}$ constitute the full list of graded projective-injective indecomposables.

\begin{prop}\label{prop-projective-criterion}
    Let $M$ be a finite-dimensional graded $\ul$-module. Suppose $M$ has a graded $\nabla$-filtration. Then $M$ is graded projective-injective if and only if it is free as a module over $\Bbbk[\dif_+]/(\dif_+^p)$.
\end{prop}
\begin{proof}
The ``only if'' part is clear from the description of indecomposable projective-injectives.

For the if part, by the theorem of Bellamy--Thiel (\cite[Theorem 5.1]{BeTh2}), it suffices to show that when a module $M$ is graded free over $\Bbbk[\dif_+]/(\dif_+^p)$, one can build a $\Delta$-filtration. To do so, choose a homogeneous vector $v$ in $M$ that generates a free $\Bbbk[\dif_+]/(\dif_+^p)$-summand that has the lowest degree, then $\dif_-(v)=0$ for degree reasons. Clearly, the $\mathfrak{sl}(2)$-submodule generated by $v$ is a baby Verma submodule. Now we can inductively build a $\Delta$-filtration by taking the quotient $M/(\ul\cdot v)$.
\end{proof}

\subsection{Web and link homology as representations}
In \cite{QRSW2,QRSW3}, an action of $\mathfrak{sl}(2)$ has been studied on foam evaluations over the base ring $\Bbbk_N:=\Bbbk[E_1,\dots, E_N]$ of symmetric functions. When $N=p$ and $\Bbbk$ is a field of finite characteristic $p>2$, the maximal ideal in $\Bbbk_N$ generated by positive degree symmetric functions is preserved under the $\mathfrak{sl}(2)$-action. The universal construction arising from foam evaluations in \cite{RW} allows us to quotient out the maximal ideal for equivariant link homology groups, and obtain a finite-dimensional $\mathfrak{gl}(p)$-Khovanov--Rozansky homology in characteristic $p$ carrying an $\mathfrak{sl}(2)$-action. Given a web or link diagram $D$, we will denote this finite-dimensional homology space by $\mH(D)$, which carries a graded $\mathfrak{sl}(2)$-structure. 

Let $t_1,t_2\in \Bbbk$ be two numbers whose sum is equal to $1$, and set $\bar{t}_i=1-t_i$ for $i=1,2$. Recall that the $\mathfrak{gl}(p)$-Khovanov--Rozansky homology is built out of resolving the two crossing diagrams:
\begin{subequations}
\begin{equation}
  T= \NB{\tikz[xscale = 0.6]{\begin{scope}[font=\tiny]
  \draw[->] (0.5, -0.5) ..controls +(0,0.3) and +(0,-0.3) .. (-0.5,
  0.5) node[pos=1, above] {} coordinate[pos =0.2] (t1);
  \fill[white] (0,0) circle (2mm);
  \draw[->] (-0.5, -0.5) ..controls +(0,0.3) and +(0,-0.3) .. (0.5,
  0.5) node[pos=1, above] {} coordinate[pos =0.2] (t2);
\end{scope}}} :=
  q^{p}\left(
    \NB{  \tikz[xscale = 3.5, yscale = 3]{
    \node (i0) at (0, 0) { $q^{-1}$ \NB{\tikz[font= \tiny,
  scale=0.6]{\begin{scope}
  \coordinate (bl) at (-0.5, -1);
  \coordinate (br) at ( 0.5, -1);
  \coordinate (tl) at (-0.5,  1);
  \coordinate (tr) at ( 0.5,  1);
    \coordinate (ml) at (-0.5,  -.8);
        \coordinate (Ml) at (-0.5,  .8);
 \coordinate (mr) at (0.5,  -.6);
\coordinate (Mr) at (0.5,  .6);
 
   \draw[>->] (bl) -- (tl) node[pos = 0, below] {$1$} node[pos = 1,
  above] {$1$} node[below, pos
  =0] {} coordinate[pos = 0.5] (g1);
  \filldraw[draw= green!50!black, fill = white] (g1) circle (1mm)
  node[left, green!50!black] {$-t_1$};
  
  
    \draw[>->] (br) -- (tr) node[pos = 0, below] {$1$} node[pos = 1, above] {$1$} node[below, pos
  =0] {} coordinate[pos = 0.5] (g2);
  \filldraw[draw= green!50!black, fill = white] (g2) circle (1mm)
  node[right, green!50!black] {$-t_2$};
  

\end{scope}}} };
     \node (i1) at (-1, 0) {\NB{\tikz[font= \tiny,
  scale=0.6]{\begin{scope}
  \coordinate (bl) at (-0.5, -1);
  \coordinate (br) at ( 0.5, -1);
  \coordinate (bm) at (  0,-0.3);
  \coordinate (tl) at (-0.5,  1);
  \coordinate (tr) at ( 0.5,  1);
  \coordinate (tm) at (  0, 0.3);
  \draw[>-]  (bl) .. controls +( 0, 0.5) and +(0,0) .. (bm)
  node[below, pos = 0] {$1$};
  \draw[>-]  (br) .. controls +( 0, 0.5) and +(0,0) .. (bm)
  node[below, pos = 0] {$1$};
  \draw[<-]  (tl) .. controls +( 0, -0.5) and +(0,0) .. (tm)
  node[above, pos = 0] {$1$} coordinate[pos = 0.25] (ga) ;

  \draw[<-]  (tr) .. controls +( 0, -0.5) and +(0,0) .. (tm)
  node[above, pos = 0] {$1$} coordinate[pos = 0.25] (gb) ;
  \draw [->-] (bm) -- (tm) node[left, pos = 0.5] {$2$};
 
\end{scope}}} };
\draw[->] (i1) -- (i0) coordinate[pos=0.5] (a);
\node[above] at (a) {\NB{\tikz[font=\tiny, scale=.5]{\begin{scope}
  \begin{scope}
    \coordinate (L1) at (0.2,0.4);
    \coordinate (L2) at (0,0);
    \coordinate (R1) at (2.2,0.4);
    \coordinate (R2) at (2,0);
    \coordinate (ML) at (0.6, 0.2);
    \coordinate (MR) at (1.6, 0.2);
    \draw[double] (ML) -- (MR);
    \draw (MR) .. controls +(0, 0) and +(-0.3,0) .. (R1) ;
    \draw (MR) .. controls +(0, 0) and +(-0.3,0) .. (R2);
    \draw (L1) .. controls +( 0.3, 0) and +(0,0) .. (ML);
    \draw (L2) .. controls +( 0.3, 0) and +(0,0) .. (ML);
  \end{scope}  
 \begin{scope}[yshift = 1cm]
    \coordinate (L1B) at (0.2,0.4);
    \coordinate (L2B) at (0,0);
    \coordinate (R1B) at (2.2,0.4);
    \coordinate (R2B) at (2,0);
    \draw (L1B) .. controls +( 0, 0) and +(0,0) .. (R1B); 
    \draw (L2B) .. controls +( 0, 0) and +(0,0) .. (R2B);
 \end{scope}  
  \draw (R1) -- (R1B);
  \draw (R2) -- (R2B);
  \draw (L1) -- (L1B);
  \draw (L2) -- (L2B);
  \draw[thick] (ML) .. controls +(0, 0.6) and +(0, 0.6) .. (MR);
\end{scope}}}};
  }}
\right) \ ,
    \end{equation}
    \begin{equation}
  T^{-1}= \NB{\tikz[xscale = 0.6]{\begin{scope}[font=\tiny]
  \draw[->] (-0.5, -0.5) ..controls +(0,0.3) and +(0,-0.3) .. (0.5,
  0.5);
  \fill[white] (0,0) circle (2mm);
  \draw[->] (0.5, -0.5) ..controls +(0,0.3) and +(0,-0.3) .. (-0.5,
  0.5);
\end{scope}}}:=
q^{-p}  \left(  \NB{\tikz[xscale = 3.5, yscale = 3]{
    \node (i0) at (-1, 0) {  $q$\ \NB{\tikz[font= \tiny,
  scale=0.6]{\begin{scope}
  \coordinate (bl) at (-0.5, -1);
  \coordinate (br) at ( 0.5, -1);
  \coordinate (tl) at (-0.5,  1);
  \coordinate (tr) at ( 0.5,  1);
    \coordinate (ml) at (-0.5,  -.8);
        \coordinate (Ml) at (-0.5,  .8);
 \coordinate (mr) at (0.5,  -.6);
\coordinate (Mr) at (0.5,  .6);
 
   \draw[>->] (bl) -- (tl) node[pos = 0, below] {$1$} node[pos = 1,
  above] {$1$} node[below, pos
  =0] {} coordinate[pos = 0.5] (g1);
  \filldraw[draw= green!50!black, fill = white] (g1) circle (1mm)
  node[left, green!50!black] {$\bar{t}_1$};
  
  
    \draw[>->] (br) -- (tr) node[pos = 0, below] {$1$} node[pos = 1, above] {$1$} node[below, pos
  =0] {} coordinate[pos = 0.5] (g2);
  \filldraw[draw= green!50!black, fill = white] (g2) circle (1mm)
  node[right, green!50!black] {$\bar{t}_2$};
  

\end{scope}}} };
     \node (i1) at (0, 0) {  \NB{\tikz[font= \tiny,
  scale=0.6]{}} };
\draw[->] (i0) -- (i1) coordinate[pos=0.5] (b);
\node[above] at (b) {\NB{\tikz[font=\tiny, scale=.5]{\begin{scope}
  \begin{scope}
    \coordinate (L1) at (0.2,0.4);
    \coordinate (L2) at (0,0);
    \coordinate (R1) at (2.2,0.4);
    \coordinate (R2) at (2,0);
    \coordinate (ML) at (0.6, 0.2);
    \coordinate (MR) at (1.6, 0.2);
    \draw[double] (ML) -- (MR);
    \draw (MR) .. controls +(0, 0) and +(-0.3,0) .. (R1) ;
    \draw (MR) .. controls +(0, 0) and +(-0.3,0) .. (R2);
    \draw (L1) .. controls +( 0.3, 0) and +(0,0) .. (ML);
    \draw (L2) .. controls +( 0.3, 0) and +(0,0) .. (ML);
  \end{scope}  
 \begin{scope}[yshift = -1cm]
    \coordinate (L1B) at (0.2,0.4);
    \coordinate (L2B) at (0,0);
    \coordinate (R1B) at (2.2,0.4);
    \coordinate (R2B) at (2,0);
    \draw (L1B) .. controls +( 0, 0) and +(0,0) .. (R1B);
    \draw (L2B) .. controls +( 0, 0) and +(0,0) .. (R2B); 
 \end{scope}  
  \draw (R1) -- (R1B);
  \draw (R2) -- (R2B);
  \draw (L1) -- (L1B);
  \draw (L2) -- (L2B);
  \draw[thick] (ML) .. controls +(0, -0.6) and +(0, -0.6) .. (MR);
\end{scope}

}}}; 
      }} 
      \right)
  \ .
    \end{equation}
  \end{subequations}
  Here in both complexes we assume (different from the convention in \cite{QRSW2}) that the terms
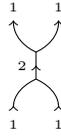
\[
  \NB{\tikz[font= \tiny,
  scale=0.6]{}} 
\]
sit in cohomological degrees $t=- 1$ and $t=1$ respectively. 

We refer the reader to \cite[Theorem 4.3]{QRSW3} to see how the action of $\mathfrak{sl}_2$ is defined in the equivariant setting and \cite[Section 6.1]{QRSW3} when the equivariant parameters are set to zero in characteristic $p$.  The green dots above indicate how the $\mathfrak{sl}_2$-action is twisted on the state spaces associated to foams.
Note that we obtain an invariant of unframed links using the convention above.

\begin{example}\label{eg-unknot}
As an $\mathfrak{sl}(2)$-module, the homology of the unknot $\ovoid$
is:
\begin{equation}
\mH(\ovoid) \cong 
\nabla(p-1) :
\quad
\xymatrix{
v_0 \ar@/^/[rr]^{\dif_+=1} && v_1 \ar@/^/[ll]^{\dif_-=p-1} \ar@/^/[rr]^{\dif_+=2}  && \cdots \ar@/^/[ll]^{\dif_-=p-2}  \ar@/^/[rr]^{\dif_+=p-2} && v_{p-2} \ar@/^/[ll]^{\dif_-=2}  \ar@/^/[rr]^{\dif_+=p-1} && v_{p-1}\ar@/^/[ll]^{\dif_-=1} 
}
\end{equation}
This is computed as in \cite[Section 7.1]{QRSW3} where one makes the necessary modifications in characteristic $p$, while also killing the equivariant parameters. More explicitly, in the language of foams (as in \cite{QRSW3}), the homology is linearly spanned by a cup with $i$ dots for $0 \leq i \leq p-1$:
\[
\mH(\ovoid)=\Bbbk\left\langle v_i=\NB{\tikz[]{\begin{scope}
  \draw (0,0) arc (180 :0: 0.5cm and 0.2cm);
  \draw[very thin] (0,0) arc (180 :0: 0.5cm and -0.6cm) node [pos=0.5,
  above] {$\bullet~i$};
  \draw (0,0) arc (180 :0: 0.5cm and -0.2cm);
\end{scope}

}}~\Big|~ 0\leq i \leq p-1\right\rangle.
\]
Recall that
\[
\nabla(p-1)\cong \Delta(p-1)=L(p-1).
\]
This module is precisely the Steinberg module for $\ul(\mathfrak{sl}(2))$ in characteristic $p$.
Note that as a vector space, it is isomorphic to the cohomology of projective space.
\end{example}

\begin{example}
    Let $\ovoid_2$ be a circle labeled by $2$.  As a vector space it is isomorphic to the cohomology of the Grassmannian of $2$-planes in $p$-dimensional space.  It is spanned by Schur polynomials $s_{\lambda}$ where $\lambda$ is a partition fitting into a Young diagram with 2 rows and $p-2$ columns. This space has an $\mathfrak{sl}(2)$-action described in \cite{QRSW2} and should serve as the $\mathfrak{gl}(p)$-Khovanov--Rozansky homology of the unknot colored by the second exterior power of the fundamental representation.  As an $\mathfrak{sl}(2)$-representation 
    \[
\mH(\ovoid_2) \cong \bigoplus_{\lambda=0}^{\lfloor \frac{p-3}{4} \rfloor} P(4 \lambda+2-2p) \ .
    \]
    The $\mathfrak{sl}(2)$-action should be related to the action of the Steenrod algebra studied in \cite{Lenart}.  See also \cite{BeCoSteenrod} for the appearance of the Steenrod algebra in categorification.

    We now provide a few details for this calculation.  As mentioned above, the vector space is spanned by a cup foam of thickness two, decorated by Schur functions $s_{m,n}$ with $p-2 \geq m \geq n \geq 0$.  The $\mathfrak{sl}(2)$-action is described in \cite[Theorem 4.11]{QRSW2}.
   We will abuse notation and leave out the presence of the foam in the basis for the state space.  The cup does twist the action of the Lie algebra.  Accounting for this twisting, using
   \cite[Lemma 6.4]{EliasQislact}, one obtains the following formulas:
   \begin{equation}
   \begin{gathered}
       \dif_+ s_{m,n} = (m+2) s_{m+1,n} + (n+1) s_{m,n+1} \ ,
       \quad \quad
       \dif_- s_{m,n} = -(m+1) s_{m-1,n} - n s_{m,n-1} \ , \\
       \dif_0 s_{m,n} = -2(2+m+n) s_{m,n} \ .
       \end{gathered}
         \end{equation}

For $m+n$ even with $m+n \leq p-2$, let
\[
v_{m,n}=\sum_{i=0}^{\frac{m+n}{2}} (-1)^i \binom{m+n+1}{i} s_{m+n-i,i} \ .
\]
It is easy to show that these are the only vectors annihilated by $\dif_-$.  It is also straightforward to prove that if $m+n \leq \frac{p-3}{2}$ then $v_{m,n}$ generates a baby Verma module
$\Delta(-4-2m-2n)$.

Assume $\frac{p-3}{2} \geq n \geq \frac{p-3}{4}$.  Modulo the baby Verma modules described above, $\dif_-$ annihilates $s_{n,n}$.
One could also show that $\dif_-^{p-1} s_{n,n} \neq 0 $, so modulo the baby Verma submodules, these $s_{n,n}$ generate baby Verma modules.
Thus the vectors $v_{i,i}$ and $s_{\frac{p-3-2i}{2},\frac{p-3-2i}{2}}$
give rise to an extension of a baby Verma module by a baby Verma.
Such an extension is unique and must be isomorphic to $P(4i+2-2p)$.

More generally, let $\ovoid_k$ be the unknot diagram labelled by an
integer $k\in \N$. When $k\in \{1,2,\dots, p-1\}$, the graded
$\ul$-module $\mH(\ovoid_k)$ is projective. This can be seen as
follows. There are projection and injection maps from the disjoint
union of $k$ unknots $\sqcup_k \ovoid$ into and out of $\ovoid_k$,
given by the cylinder
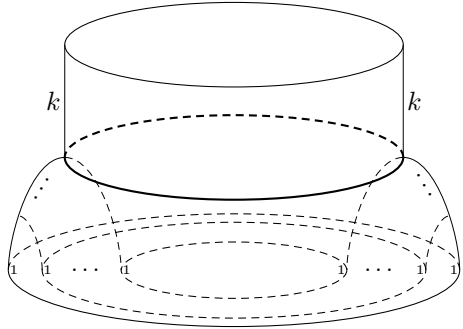
\[
\NB{\tikz[]{\begin{scope}[scale=1.5]
  \draw  (2, 0) arc (0:-180:2cm and 0.5cm);
  \draw[densely dashed] (2,0) arc (0:180:2cm and 0.5cm);
  \draw[densely dashed] (0,0) circle (1.7cm and 0.425cm);
  \node[font=\small] at (1.3, 0) {$\dots$};
  \node[font =\small] at (-1.3, 0) {$\dots$};
  \node[font=\tiny] at (1.95,0) {$1$};
  \node[font=\tiny] at (1.65,0) {$1$};
  \node[font=\tiny] at (0.95,0) {$1$};
  \node[font=\tiny] at (-1.95,0) {$1$};
  \node[font=\tiny] at (-1.65,0) {$1$};
  \node[font=\tiny] at (-0.95,0) {$1$};
  \node[font=\normalsize] at (1.6,1.5) {$k$};
  \node[font=\normalsize] at (-1.6,1.5) {$k$};
  \draw[densely dashed] (0,0) circle (1cm and 0.25cm);
  \draw[thick] (1.5,1) arc (0:-180: 1.5cm and 0.375cm);
  \draw[densely dashed, thick] (1.5,1) arc (0:180: 1.5cm and 0.375cm);
  \draw (0,2) circle (1.5cm and 0.375cm);
  \begin{scope}[very thin]
  \draw (1.5, 1) .. controls +(0.3, 0) and +(0, 0.3) .. (2, 0)
  coordinate[pos=0.6] (a) node[pos=0.35, sloped, font=\small, below] {$\dots$}; 
  \draw[densely dashed] (1.7,0) .. controls +(0,0.2) and +(-0.1, -0) .. (a);
  \draw (-1.5, 1) .. controls +(-0.3, 0) and +(0, 0.3) .. (-2, 0)
  coordinate[pos=0.6] (b) node[pos=0.35, sloped, font=\small, below] {$\dots$}; 
  \draw[densely dashed] (-1.7,0) .. controls +(0,0.2) and +(0.1, -0) .. (b); 
  \draw[densely dashed] (1.5, 1) .. controls +(-0.3, 0) and +(0, 0.3) .. (1, 0);
  \draw[densely dashed] (-1.5, 1) .. controls +(0.3, 0) and
  +(0, 0.3) .. (-1, 0); 
  \draw (1.5, 1) -- +(0,1);
  \draw (-1.5, 1) -- +(0,1);
  \end{scope}
\end{scope}}}
\]
and its upside-down image. The composition of these maps is equal to $k!$, which is nonzero if $k\in \{1,\dots, p-1\}$. Furthermore, these foams are acted upon by $\ul$ trivially using the $\mathfrak{{sl}}(2)$-action in \cite{QRSW3}. More precisely, as in \cite{QRSW3}, $\dif_-$ acts trivially on trivially-decorated foams, $\dif_0$ acts trivially on the cylindrical foam thanks to equations (25) and (28), while $\dif_+$ acts trivially on it because of equations (34) and (36).
It follows that $\mH(\ovoid_k)$ is a $\ul$-direct summand of $\mH(\ovoid)^{\otimes k}$ whenever $k\in \{1,\dots, p-1\}$.  Since $\mH(\sqcup_k\ovoid)$ is the $k$-fold tensor product of the graded projective-injective Steinberg module, it is also projective-injective and so is any of its direct summands. Notice that, for $k \in \{1,\dots, p-1\}$, the symmetric group algebra $\Bbbk[S_k]$ is semisimple and commutes with the action of $\ul$. Using this, it is an easy exercise to see that one can identify $\mH(\ovoid_k)$ with the summand $\wedge^k \mH(\ovoid)$, the degree $k$ exterior tensor power of the Steinberg module.
\end{example}

\begin{example}
    Let $\Theta$ denote the theta graph, that is the graph with two
    vertices and three edges labeled by $1$, $1$ and $2$.
    Then for $t_1=1,t_2=0$
    \begin{equation} \label{eq:thetarep}
    \mH(\Theta) \cong L(1) \otimes \left(\bigoplus_{\lambda=0}^{\lfloor \frac{p-3}{4} \rfloor} P(4 \lambda+2-2p)\right) \ .
    \end{equation}
    The second tensor factor corresponds to the representation of the circle colored by $2$ in the previous example.  Note that this is a projective-injective module.

When $p=3$ and $t_1=t_2=2$, the operator $\dif_+$ annihilates the highest weight vector.  This precludes the possibility of $ \mH(\Theta) $ having a baby Verma filtration and thus it is not projective (see Proposition \ref{prop-projective-criterion}).  This small example shows that in general, the state space of a web need not be a graded projective $\ul$-module.
\end{example}

\begin{prop}\label{prop-dual-Verma-filtration}
   For a link diagram $D$, both the $\mathfrak{gl}(p)$-Khovanov--Rozansky chain complex and the homology of $D$ in characteristic $p$ have finite-step filtrations whose subquotients are dual baby Verma modules.
\end{prop}
\begin{proof}
    Choose a base point $b$ on $D$, so that $\mC(D)$ and $\mH(D)$ are graded modules over $A=\Bbbk[x_b]/(x_b^p)$.  By Proposition \ref{prop-unreduced-reduced}$, \mC(D)$ and $\mH(D)$ are free modules over $A$, so that there are identifications
    \[
    \mC(D)\cong   A \otimes \bar{\mC}(D) ,\quad \quad \mH(D)\cong A \otimes \bar{\mH}(D) .
    \]
    For simplicity, we will carry out the rest of the argument for $\mC(D)$, as the case for $\mH(D)$ is similar.

    Identify the subcomplex $\bar{\mC}(D)$ with $x_b^{p-1}\cdot \mC(D)$. Then $\bar{\mC}(D)$ carries an action of $\mathfrak{b}_+$. Filter $x_b^{p-1}\mC(D)$ as $\mathfrak{b}_+$-modules so that the subquotients are 1-dimensional $\mathfrak{b}_+$-modules. Then tensoring this filtration with $A$ gives rise to a filtration of $\mC(D)$ whose subquotients are isomorphic to $A$ as $\dif_-$-modules. Since $\dif_-$ acted upon $x_b^{p-1}$ generates the entire $A$, it follows that the subquotients are dual baby Vermas. The desired result follows.
\end{proof}

\begin{cor}
    With respect to the operator $\dif_-$, the homology $\mH(D)$ is acyclic as a $p$-DG module.
\end{cor}

\begin{proof}
By Proposition \ref{prop-dual-Verma-filtration}, $\mH(D)$ has a filtration with subquotients dual baby Verma modules.  Each such module is acyclic.
\end{proof}

 For a link diagram $D$, denote by $D^!$ its mirror image in $\R^3$. 
\begin{prop}
    For any link $D$, there is an isomorphism of graded $\mathfrak{sl}(2)$-modules $\mH(D)\cong \mH(D^!)^*$.
\end{prop}
\begin{proof}
      Form the tube $D\times [0,1]$ inside $\R^4$, and bend it into the shapes of $D\times \cup$ and $D\times \cap$. Since these tubes are disjoint unions of (tangled) cylinders, they induce $\mathfrak{sl}(2)$-commuting pairings and copairings between $\mH(D)$ and $\mH(D^!)$. The result follows.
\end{proof}
\begin{example}
When $L$ is the Hopf link $T_{2,2}$, as an $\mathfrak{sl}(2)$-module, the homology of $L$ is given by
\[
\mH(T_{2,2}) \cong
q^{-2}\nabla(0) \oplus t^{-2}( (\nabla(2t_1-3) \otimes \nabla(2t_2-3)) / \nabla(-2))
\]
where the second summand is in cohomological degree $-2$, indicated by the cohomological shift $t^{-2}$.
Here, $\nabla(-2)$ is embedded in $\nabla(2t_1-3) \otimes \nabla(2t_2-3)$
by mapping $v_k$ to $\sum_{i+j=p-1+k} v_i \otimes v_j$.
This computation is done as in \cite{Roz} for the $N=2$ case, but in characteristic $p$, and setting the equivariant parameters to zero.
The chain complex simplifies as in \cite[Proposition 5.1]{Roz}.
In cohomological degree zero, the space is a quotient of a space spanned by two cup foams with dots.  The $\mathfrak{sl}(2)$-action is computed using the formulas in \cite[Section 3.2]{QRSW3}.
In cohomological degree -2, the space is a quotient of the space spanned by a half-theta foam with dots on the thin facets.  The dimension of the space is $p(p-1)$ (which is the decategorified evaluation of the theta web).
The $\mathfrak{sl}(2)$-action is again computed using the formulas in \cite[Section 3.2]{QRSW3}.  Decomposing the spanning foams into basic foams, one has an unzip foam and two cup foams.

\end{example}
\begin{example}
    More generally, if $L=T_{2,n}$ is a $(2,n)$-torus link, one could compute the homology using \cite{Roz} and \cite{KRWitt}. 

    If $n$ is odd, then
    \begin{equation}
    \mH(L)\cong q^{-n} \nabla(0) \oplus
    \bigoplus_{j=1}^{\frac{n-1}{2}} \left( t^{-2j} q^{-n+4j-1} \nabla(2t_1-2j+p-1) \oplus
    t^{-2j-1} q^{-n+4j+1} \nabla(2-2j) 
    \right)
    \ .
    \end{equation}
      If $n$ is even, then
     \begin{align}
    \mH(L)\cong q^{-n} \nabla(0) &\oplus
    \bigoplus_{j=1}^{\frac{n-1}{2}} \left( t^{-2j} q^{-n+4j-1} \nabla(2t_1-2j+p-1) \oplus
    t^{-2j-1} q^{-n+4j+1} \nabla(2-2j) 
    \right)\nonumber \\
    &\oplus
    t^{-n} (\nabla(2t_1-n+p-1) \otimes \nabla(2t_2-n+p-1)) / \nabla(2-2n)
    \ .
    \end{align}

    Note that the homology is different than in the characteristic 0 calculation of \cite[Section 10.3]{CautisClasp}.
\end{example}

\subsection{Some topological applications}

The state spaces of webs and links enjoy certain parity-unimodality properties as a consequence of the $\mathfrak{sl}_2$-symmetry.
Recall that we choose the parameters $t_1, t_2$ in the $\mathfrak{sl}_2$-action to satisfy $t_1+t_2=1$ so that the $\dif_0$ operator acts as negative of the degree operator on foams.

For any web or link $D$ diagram, consider its state space $\mH_{j,i}(D)$ in a particular cohomological degree $j$ and  $q$-degree $i$.  

The bigraded dimension (Poincar\'{e} polynomial) of the $\mathfrak{gl}(p)$-homology theory
\[
\gdim \mH(D) = \sum_{i, j \in \mathbb{Z}}t^j \dim_\Bbbk \mH_{j,i}(D) q^i \in \Z[t^{\pm 1},q^{\pm 1}].
\]
is an invariant of $D$, which specializes to the $\mathfrak{gl}(p)$-polynomial of $D$ when $t=-1$. Now we consider this polynomial in the quotient ring $\Z[t^{\pm 1},q]/(q^{2p}-1)$, and denote the image polynomial by $\gdim_p \mH(D)$. In effect, we have
\begin{align}
\gdim_p \mH(D) & = \sum_{ j \in \mathbb{Z}}\sum_{i=0}^{2p-1} t^j \dim_\Bbbk \left(\bigoplus_{k\in \mathbb{Z}} \mH_{j,i+2kp}(D)\right) q^{i} \nonumber \\
& = \sum_{j\in \Z} \sum_{i=0}^{2p-1} \lambda_{j,i}t^jq^i \label{eq:gdim}
\end{align}

\begin{thm}
Let $D$ be  a link or web diagram which has an edge of thickness 1.
   For any $j\in \Z$, the coefficients of $ \gdim_p \mH(D)$ in \eqref{eq:gdim} satisfy a parity unimodality property:
        \[
    \lambda_{j,0} = \lambda_{j,2} = \ldots = \lambda_{j,2p-2}\ , \quad \quad
    \lambda_{j,1} = \lambda_{j,3} = \ldots = \lambda_{j,2p-1} \ .
    \]
\end{thm}

\begin{proof}
This follows directly from Proposition \ref{prop-dual-Verma-filtration}.  For such a diagram $D$, the homology of $D$ has a dual baby Verma filtration, whose non-zero weight spaces are of equal dimension one.
\end{proof}

The rigidity of the Steinberg module provides an obstruction criterion to the sliceness of a knot.

\begin{thm} \label{thm:stein}
    Let $K$ be an oriented knot. If $K$ is slice, then the $\mathfrak{gl}(p)$-homology of $K$ contains a Steinberg summand in cohomological degree 0.
\end{thm}
\begin{proof}
An embedded cylinder cobordism in $\R^4$ from the unknot to $K$ defines an $\mathfrak{sl}(2)$-commuting homomorphism from the $\mathfrak{gl}(p)$-homology of the unknot to that of $K$. As for the usual Khovanov homology and its Lee deformations, one can verify that this map is nonzero. Then, as seen in Example \ref{eg-unknot}, the homology of the unknot is the Steinberg module, which is both simple, projective and injective as a graded $\mathfrak{sl}(2)$-module.  It follows that the homomorphism is injective, and the image of the unknot homology is a direct summand of $\mH(L)$ because of graded injectivity.
\end{proof}
 It is a natural question to ask about the converse: Is the condition in Theorem \ref{thm:stein} an ``if and only if'' condition? 

Next, motivated by Lipshitz--Sarkar \cite{LipSarSplit} and Wang \cite{WangJ2}, we show that the number of copies of $\mathfrak{sl}(2)$ acting on the $\mathfrak{gl}(p)$-homology detects the splitness of a link. Suppose $L=L_1\sqcup L_2$ is the split disjoint union of two links $L_1$ and $L_2$. In other words,  $L_1$ and $L_2$ are separated by an embedded sphere in $\R^3$ away from $L$. Then each $\mH(L_i)$ carries an independent $\mathfrak{sl}(2)$-action, and thus $\mH(L)=\mH(L_1)\otimes \mH(L_2)$ carries an action of $\mathfrak{sl}(2)\times \mathfrak{sl}(2)$. The earlier $\mathfrak{sl}(2)$-action on $\mH(L)$ comes from diagonally embedding $\mathfrak{sl}(2)$ into $\mathfrak{sl}(2)\times \mathfrak{sl}(2)$. 

Now choose base points on $L_i$ so that $\mH(L_i)$ are $\Bbbk[x_i]/(x_i^p)$-modules, $i=1,2$. To differentiate the two copies of $\mathfrak{sl}(2)$ acting on $\mH(L_i)$,
we will call the Lie algebra generators $\dif_{i,s}$, where $i=1,2$ and $s\in \{0,\pm\}$. The operators satisfy
\begin{equation}\label{eqn-2-sl(2)}
\dif_{i,-}(x_j)=-\delta_{i,j}, \quad \quad  \quad \dif_{i,0}(x_j)=-\delta_{i,j} 2x_j, \quad \quad \quad \dif_{i,+}(x_j)=\delta_{i,j} x_j^2,\quad \quad \quad i,j\in \{1,2\}.
\end{equation}
The converse statement also holds.
 
\begin{thm}
    Consider a link $L$ where $x_i$ are base points contained in two distinct components $L_i$, $i=1,2$. Then there exists an $\mathfrak{sl}_2\times \mathfrak{sl}_2$-action on the $\mathfrak{gl}(p)$-Khovanov--Rozansky homology of $L$ satisfying \eqref{eqn-2-sl(2)} if and only if $L_1$ and $L_2$ are separated by an embedded sphere in $\R^3$ away from $L$.
\end{thm}
\begin{proof}
    This is a consequence of Wang's theorem \cite[Theorem 4.13]{WangJ2} on $\mH(L)$ detecting splitness of links. Wang's theorem shows that, as a module over the truncated polynomials rings $\Bbbk[x_1,x_2]/(x_1^p, x_2^p)$, $\mH(L)$ is free if and only if the components $L_i$'s are separated by an embedded sphere in $\R^3$ away from $L$.

    Clearly, if the $L_i$'s are separated by an embedded sphere in $\R^3$ away from $L$, then there is an action of the Lie algebra $\mathfrak{sl}_2\times \mathfrak{sl}_2$ on the homology $\mH(L)$, where the two factors act on the homology of the link components inside and outside of the embedded sphere respectively.

    Conversely, suppose there is an $\mathfrak{sl}_2\times \mathfrak{sl}_2$-action on $\mH(L)$ such that
    \[
    \dif_{i,-}(x_j)=\delta_{i,j}, \quad \quad i,j=1,2. 
    \]
    Then $\mH(L)$ is a module algebra over 
    \[
    (A_1\otimes A_2)\# \dfrac{\Bbbk[\dif_{1,-},\dif_{2,-}]}{(\dif_{1,-}^p,\dif_{2,-}^p)} \cong \mathrm{M}_p(\Bbbk)\otimes \mathrm{M}_p(\Bbbk).
    \]
    As modules over the last tensor product matrix algebra are always free over $\Bbbk[x_1,x_2]/(x_1^p, x_2^p)$, the result follows from Wang's theorem.
\end{proof}

\addcontentsline{toc}{section}{References}

\bibliographystyle{alphaurl}
\bibliography{extracted}

\noindent Y.~Q.: { \sl \small Department of Mathematics, University of Virginia, Charlottesville, VA 22904, USA} \newline \noindent {\tt \small email: yq2dw@virginia.edu}

\vspace{0.1in}

\noindent L.-H.~R.:  {\sl \small Université Clermont Auvergne, LMBP,
    Campus des Cézeaux, 3 place Vasarely, TSA 60026, CS 60026, 63178
    Aubière Cedex, France}
\newline \noindent {\tt \small email: louis-hadrien.robert@uca.fr  }

\vspace{0.1in}

\noindent J.~S.:  {\sl \small Department of Mathematics, CUNY Medgar Evers, Brooklyn, NY, 11225, USA}\newline \noindent {\tt \small email: jsussan@mec.cuny.edu \newline 
\sl \small Mathematics Program, The Graduate Center, CUNY, New York, NY, 10016, USA}

\vspace{0.1in}

\noindent E.~W.: { \sl \small Univ Paris Diderot, IMJ-PRG, UMR 7586 CNRS, F-75013, Paris, France} \newline \noindent {\tt \small email: wagner@imj-prg.fr}

\end{document}